\documentclass[11pt, twoside]{article}
\usepackage{latexsym}
\usepackage{cite}
\usepackage{amsmath}
\usepackage{amssymb}
\usepackage[all]{xy}
\usepackage{amsfonts}
\usepackage{verbatim}
\usepackage{amsthm}
\usepackage{mathrsfs}
\usepackage{epsfig}
\usepackage{xy}
\usepackage{array}
\usepackage{stmaryrd}
\usepackage{graphicx,color}
\usepackage{xcolor}
\usepackage[colorlinks=true,linkcolor=blue,citecolor=blue]{hyperref}
\usepackage{tikz}
\usetikzlibrary{arrows,calc}
\usepackage{etex}
\usepackage{mathdots}
\usepackage{float}
\usepackage{graphics}
\usepackage{pdflscape}
\usepackage{extarrows}
\usepackage{anysize,hyperref}
\input xypic
\xyoption{all}
\usepackage{bm}
\usepackage[perpage,symbol]{footmisc}
\usepackage{setspace}
\usepackage{enumitem}

\def\A{\mathscr{A}}

\def\C{\mathscr{C}}\def\H{\mathscr{H}}
\def\E{\mathbb{E}}
\def\F{\mathbb{F}}
\def\s{\mathfrak{s}}
\def\id{\mathrm{id}}
\def\op{^\mathrm{op}}
\def\Ab{\mathsf{Ab}}
\def\del{\delta}
\def\dr{\ar@{->}[r]}

\newcommand{\CC}{{\bf{C}}^{n+2}_{\C}}

\newcommand{\ov}{\overset}
\newcommand{\lra}{\longrightarrow}
\newcommand{\co}{\colon}
\newcommand{\uas}{^{\ast}}            %%% ^*
\newcommand{\sas}{_{\ast}}
\newcommand{\Xd}{\langle X_{\bullet},\del\rangle}  %%% <X,¦Ä>
\newcommand{\ush}{^\sharp}           %%% ^sharp
\newcommand{\ssh}{_\sharp}
\SelectTips{cm}{10}

\usepackage{setspace}

\begin{document}
\baselineskip=15pt
\title{\Large{\bf Two results of $n$-exangulated categories\footnotetext{~Jian He was supported by the National Natural Science Foundation of China (Grant No. 12171230). Panyue Zhou was supported by the National Natural Science Foundation of China (Grant No. 11901190).} }}
\medskip
\author{Jian He, Jing He and Panyue Zhou}

\date{}

\maketitle
\def\blue{\color{blue}}
\def\red{\color{red}}

\newtheorem{theorem}{Theorem}[section]
\newtheorem{lemma}[theorem]{Lemma}
\newtheorem{corollary}[theorem]{Corollary}
\newtheorem{proposition}[theorem]{Proposition}
\newtheorem{conjecture}{Conjecture}
\theoremstyle{definition}
\newtheorem{definition}[theorem]{Definition}
\newtheorem{question}[theorem]{Question}
\newtheorem{remark}[theorem]{Remark}
\newtheorem{remark*}[]{Remark}
\newtheorem{example}[theorem]{Example}
\newtheorem{example*}[]{Example}
\newtheorem{condition}[theorem]{Condition}
\newtheorem{condition*}[]{Condition}
\newtheorem{construction}[theorem]{Construction}
\newtheorem{construction*}[]{Construction}

\newtheorem{assumption}[theorem]{Assumption}
\newtheorem{assumption*}[]{Assumption}

\baselineskip=17pt
\parindent=0.5cm

\begin{abstract}
\begin{spacing}{1.2}
$n$-exangulated categories were introduced by Herschend-Liu-Nakaoka which are a simultaneous generalization of $n$-exact categories and $(n+2)$-angulated categories. This paper consists of two results on $n$-exangulated categories: (1) we  give an equivalent characterization of the axiom (EA2); (2) we provide a new way to construct a closed subfunctor of an $n$-exangulated category.\\[0.2cm]
 \end{spacing}
\medskip
\end{abstract}

\pagestyle{myheadings}
\markboth{\rightline {\scriptsize J. He, J. He and P. Zhou }}
         {\leftline{\scriptsize Two results of $n$-exangulated categories}}

\section{Introduction}
The notion of extriangulated categories was introduced in \cite{NP}, which can be viewed as a simultaneous generalization of exact categories and triangulated categories.
 The data of such a category is a triplet $(\C,\E,\s)$, where $\C$ is an additive category, $\mathbb{E}\colon \C^{\rm op}\times \C \rightarrow \Ab$ is an additive bifunctor and $\mathfrak{s}$ assigns to each $\delta\in \mathbb{E}(C,A)$ a class of $3$-term sequences with end terms $A$ and $C$ such that certain axioms hold. Recently, Herschend--Liu--Nakaoka \cite{HLN}
introduced the notion of $n$-exangulated categories for any positive integer $n$. It should be noted that the case $n=1$ corresponds to
extriangulated categories. As typical examples we know that $n$-exact categories and $(n+2)$-angulated categories are $n$-exangulated categories, see \cite[Proposition 4.34]{HLN} and \cite[Proposition 4.5]{HLN}. However, there are some other examples of $n$-exangulated categories which are neither $n$-exact nor $(n+2)$-angulated, see \cite{HLN, HLN1, LZ,HZZ2}.

Recall that homotopy cartesian squares in triangulated categories are the triangulated analogues of the pushout and pullback squares in abelian categories. It is well-known that the axiom (TR4) is equivalent to the homotopy cartesian axiom. Recently, Kong-Lin-Wang \cite{KLW} introduced the notion of homotopy cartesian squares
and proposed a new shifted octahedron in extriangulated categories. Our first main result show that Kong-Lin-Wang's result has a higher counterpart:
\begin{theorem}\rm{ (see Theorem \ref{thm1} for details)}
Let $(\C,\E,\s)$ be a pre-$n$-exangulated category. Then $\C$ satisfies {\rm(EA2)} if and only if $\C$ satisfies {\rm(EA2-1)}.
\end{theorem}

Relative theories are explicated by use of closed subfunctors in $n$-exangulated categories \cite{HLN}. An additive subfunctor $\F$ of $\E$ is a closed subfunctor if $(\C,\F,\s|_\F)$ is an $n$-exangulated category,  where $\s|_\F$ is a restriction of $\s$ to $\F$, see \cite[Proposition 3.16]{HLN}. Our second main result provide a new way to construct a closed subfunctor of an $n$-exangulated category. This is a higher counterpart of Sakai's result.

\begin{theorem}\rm{ (see Theorem \ref{th1} for details)}
Let $(\C, \E, \s)$ be an $n$-exangulated category and $F\colon\C\to\A$ a half exact functor to an abelian category $\A$.
Then there exists a unique maximal closed subfunctor $\F$ of $\E$ such that $F$ becomes a right exact functor from $(\C, \F, \s|_{\F})$ to $\A$.
\end{theorem}

This article is organized as follows. In Section 2, we review some elementary definitions and facts on $n$-exangulated categories. In Section 3, we prove our first main result. In Section 4, we prove our second main result.

\section{Preliminaries}
In this section, we briefly review basic concepts and results concerning $n$-exangulated categories.

 Let $\C$ be an additive category and
 $\mathbb{E}\colon \C^{\rm op}\times \C \rightarrow \Ab~~\mbox{($\Ab$ is the category of abelian groups)}$
 an additive bifunctor.
 {For any pair of objects $A,C\in\C$, an element $\del\in\E(C,A)$ is called an {\it $\E$-extension} or simply an {\it extension}. We also write such $\del$ as ${}_A\del_C$ when we indicate $A$ and $C$. The zero element ${}_A0_C=0\in\E(C,A)$ is called the {\it split $\E$-extension}. For any pair of $\E$-extensions ${}_A\del_C$ and ${}_{A'}\del{'}_{C'}$, let $\delta\oplus \delta'\in\mathbb{E}(C\oplus C', A\oplus A')$ be the
element corresponding to $(\delta,0,0,{\delta}{'})$ through the natural isomorphism $\mathbb{E}(C\oplus C', A\oplus A')\simeq\mathbb{E}(C, A)\oplus\mathbb{E}(C, A')
\oplus\mathbb{E}(C', A)\oplus\mathbb{E}(C', A')$.

For any $a\in\C(A,A')$ and $c\in\C(C',C)$,  $\E(C,a)(\del)\in\E(C,A')\ \ \text{and}\ \ \E(c,A)(\del)\in\E(C',A)$ are simply denoted by $a_{\ast}\del$ and $c^{\ast}\del$, respectively.

Let ${}_A\del_C$ and ${}_{A'}\del{'}_{C'}$ be any pair of $\E$-extensions. A {\it morphism} $(a,c)\colon\del\to{\delta}{'}$ of extensions is a pair of morphisms $a\in\C(A,A')$ and $c\in\C(C,C')$ in $\C$, satisfying the equality
$a_{\ast}\del=c^{\ast}{\delta}{'}$.}

Let $\C$ be an additive category as before, and let $n$ be any positive integer.
\begin{definition}\cite[Definition 2.7]{HLN}
Let $\bf{C}_{\C}$ be the category of complexes in $\C$. As its full subcategory, define $\CC$ to be the category of complexes in $\C$ whose components are zero in the degrees outside of $\{0,1,\ldots,n+1\}$. Namely, an object in $\CC$ is a complex $X_{\bullet}=\{X_i,d^X_i\}$ of the form
\[ X_0\xrightarrow{d^X_0}X_1\xrightarrow{d^X_1}\cdots\xrightarrow{d^X_{n-1}}X_n\xrightarrow{d^X_n}X_{n+1}. \]
We write a morphism $f_{\bullet}\co X_{\bullet}\to Y_{\bullet}$ simply $f_{\bullet}=(f_0,f_1,\ldots,f_{n+1})$, only indicating the terms of degrees $0,\ldots,n+1$.
\end{definition}

\begin{definition}\cite[Definition 2.11]{HLN}
By Yoneda lemma, any extension $\del\in\E(C,A)$ induces natural transformations
\[ \del\ssh\colon\C(-,C)\Rightarrow\E(-,A)\ \ \text{and}\ \ \del\ush\colon\C(A,-)\Rightarrow\E(C,-). \]
For any $X\in\C$, these $(\del\ssh)_X$ and $\del\ush_X$ are given as follows.
\begin{enumerate}
\item[\rm(1)] $(\del\ssh)_X\colon\C(X,C)\to\E(X,A)\ :\ f\mapsto f\uas\del$.
\item[\rm (2)] $\del\ush_X\colon\C(A,X)\to\E(C,X)\ :\ g\mapsto g\sas\delta$.
\end{enumerate}
We simply denote $(\del\ssh)_X(f)$ and $\del\ush_X(g)$ by $\del\ssh(f)$ and $\del\ush(g)$, respectively.
\end{definition}

\begin{definition}\cite[Definition 2.9]{HLN}
 Let $\C,\E,n$ be as before. Define a category $\AE:=\AE^{n+2}_{(\C,\E)}$ as follows.
\begin{enumerate}
\item[\rm(1)] A object in $\AE^{n+2}_{(\C,\E)}$ is a pair $\Xd$ of $X_{\bullet}\in\CC$
and $\del\in\E(X_{n+1},X_0)$ satisfying
$$(d_0^X)_{\ast}\del=0~~\textrm{and}~~(d^X_n)^{\ast}\del=0.$$
We call such a pair an $\E$-attached
complex of length $n+2$. We also denote it by
$$X_0\xrightarrow{d_0^X}X_1\xrightarrow{d_1^X}\cdots\xrightarrow{d_{n-2}^X}X_{n-1}
\xrightarrow{d_{n-1}^X}X_n\xrightarrow{d_n^X}X_{n+1}\overset{\delta}{\dashrightarrow}.$$
\item[\rm (2)]  For such pairs $\Xd$ and $\langle Y_{\bullet},\rho\rangle$, a morphism $f_{\bullet}\colon\Xd\to\langle Y_{\bullet},\rho\rangle$ is
defined to be a morphism $f_{\bullet}\in\CC(X_{\bullet},Y_{\bullet})$ satisfying $(f_0)_{\ast}\del=(f_{n+1})^{\ast}\rho$.

We use the same composition and the identities as in $\CC$.

\end{enumerate}
\end{definition}

\begin{definition}\cite[Definition 2.13]{HLN}\label{def1}
 An {\it $n$-exangle} is a pair $\Xd$ of $X_{\bullet}\in\CC$
and $\del\in\E(X_{n+1},X_0)$ which satisfies the following conditions.
\begin{enumerate}
\item[\rm (1)] The following sequence of functors $\C\op\to\Ab$ is exact.
$$
\C(-,X_0)\xrightarrow{\C(-,\ d^X_0)}\cdots\xrightarrow{\C(-,\ d^X_n)}\C(-,X_{n+1})\xrightarrow{~\del\ssh~}\E(-,X_0)
$$
\item[\rm (2)] The following sequence of functors $\C\to\Ab$ is exact.
$$
\C(X_{n+1},-)\xrightarrow{\C(d^X_n,\ -)}\cdots\xrightarrow{\C(d^X_0,\ -)}\C(X_0,-)\xrightarrow{~\del\ush~}\E(X_{n+1},-)
$$
\end{enumerate}
In particular any $n$-exangle is an object in $\AE$.
A {\it morphism of $n$-exangles} simply means a morphism in $\AE$. Thus $n$-exangles form a full subcategory of $\AE$.
\end{definition}

\begin{definition}\cite[Definition 2.22]{HLN}
Let $\s$ be a correspondence which associates a homotopic equivalence class $\s(\del)=[{}_A{X_{\bullet}}_C]$ to each extension $\del={}_A\del_C$. Such $\s$ is called a {\it realization} of $\E$ if it satisfies the following condition for any $\s(\del)=[X_{\bullet}]$ and any $\s(\rho)=[Y_{\bullet}]$.
\begin{itemize}
\item[{\rm (R0)}] For any morphism of extensions $(a,c)\co\del\to\rho$, there exists a morphism $f_{\bullet}\in\CC(X_{\bullet},Y_{\bullet})$ of the form $f_{\bullet}=(a,f_1,\ldots,f_n,c)$. Such $f_{\bullet}$ is called a {\it lift} of $(a,c)$.
\end{itemize}
In such a case, we simply say that \lq\lq$X_{\bullet}$ realizes $\del$" whenever they satisfy $\s(\del)=[X_{\bullet}]$.

Moreover, a realization $\s$ of $\E$ is said to be {\it exact} if it satisfies the following conditions.
\begin{itemize}
\item[{\rm (R1)}] For any $\s(\del)=[X_{\bullet}]$, the pair $\Xd$ is an $n$-exangle.
\item[{\rm (R2)}] For any $A\in\C$, the zero element ${}_A0_0=0\in\E(0,A)$ satisfies
\[ \s({}_A0_0)=[A\ov{\id_A}{\lra}A\to0\to\cdots\to0\to0]. \]
Dually, $\s({}_00_A)=[0\to0\to\cdots\to0\to A\ov{\id_A}{\lra}A]$ holds for any $A\in\C$.
\end{itemize}
Note that the above condition {\rm (R1)} does not depend on representatives of the class $[X_{\bullet}]$.
\end{definition}

\begin{definition}\cite[Definition 2.23]{HLN}
Let $\s$ be an exact realization of $\E$.
\begin{enumerate}
\item[\rm (1)] An $n$-exangle $\Xd$ is called an $\s$-{\it distinguished} $n$-exangle if it satisfies $\s(\del)=[X_{\bullet}]$. We often simply say {\it distinguished $n$-exangle} when $\s$ is clear from the context.
\item[\rm (2)]  An object $X_{\bullet}\in\CC$ is called an {\it $\s$-conflation} or simply a {\it conflation} if it realizes some extension $\del\in\E(X_{n+1},X_0)$.
\item[\rm (3)]  A morphism $f$ in $\C$ is called an {\it $\s$-inflation} or simply an {\it inflation} if it admits some conflation $X_{\bullet}\in\CC$ satisfying $d_0^X=f$.
\item[\rm (4)]  A morphism $g$ in $\C$ is called an {\it $\s$-deflation} or simply a {\it deflation} if it admits some conflation $X_{\bullet}\in\CC$ satisfying $d_n^X=g$.
\end{enumerate}
\end{definition}

\begin{definition}\cite[Definition 2.27]{HLN}
For a morphism $f_{\bullet}\in\CC(X_{\bullet},Y_{\bullet})$ satisfying $f_0=\id_A$ for some $A=X_0=Y_0$, its {\it mapping cone} $M_{_{\bullet}}^f\in\CC$ is defined to be the complex
\[ X_1\xrightarrow{d^{M_f}_0}X_2\oplus Y_1\xrightarrow{d^{M_f}_1}X_3\oplus Y_2\xrightarrow{d^{M_f}_2}\cdots\xrightarrow{d^{M_f}_{n-1}}X_{n+1}\oplus Y_n\xrightarrow{d^{M_f}_n}Y_{n+1} \]
where $d^{M_f}_0=\begin{bmatrix}-d^X_1\\ f_1\end{bmatrix},$
$d^{M_f}_i=\begin{bmatrix}-d^X_{i+1}&0\\ f_{i+1}&d^Y_i\end{bmatrix}\ (1\le i\le n-1),$
$d^{M_f}_n=\begin{bmatrix}f_{n+1}&d^Y_n\end{bmatrix}$.

{\it The mapping cocone} is defined dually, for morphisms $h_{\bullet}$ in $\CC$ satisfying $h_{n+1}=\id$.
\end{definition}

\begin{definition}\cite[Definition 2.32]{HLN}
An {\it pre-$n$-exangulated category} is a triplet $(\C,\E,\s)$ of additive category $\C$, additive bifunctor $\E\co\C\op\times\C\to\Ab$, and its exact realization $\s$, satisfying the following condition.
\begin{itemize}[leftmargin=3.3em]
\item[{\rm (EA1)}] Let $A\ov{f}{\lra}B\ov{g}{\lra}C$ be any sequence of morphisms in $\C$. If both $f$ and $g$ are inflations, then so is $g\circ f$. Dually, if $f$ and $g$ are deflations, then so is $g\circ f$.
\end{itemize}

If the triplet $(\C,\E,\s)$  moreover satisfies the following conditions, then it is called an {\it $n$-exangulated category}
\begin{itemize}[leftmargin=3.3em]
\item[{\rm (EA2)}] For $\rho\in\E(D,A)$ and $c\in\C(C,D)$, let ${}_A\langle X_{\bullet},c\uas\rho\rangle_C$ and ${}_A\langle Y_{\bullet},\rho\rangle_D$ be distinguished $n$-exangles. Then $(\id_A,c)$ has a {\it good lift} $f_{\bullet}$, in the sense that its mapping cone gives a distinguished $n$-exangle $\langle M^f_{_{\bullet}},(d^X_0)\sas\rho\rangle$.
\end{itemize}   
\begin{itemize}[leftmargin=4.2em]
 \item[{\rm (EA2$\op$)}] Dual of {\rm (EA2)}.
\end{itemize}
 Note that in the case $n=1$, a triplet $(\C,\E,\s)$ is a  $1$-exangulated category if and only if it is an extriangulated category, see \cite[Proposition 4.3]{HLN}.
\end{definition}

\begin{example}
From \cite[Proposition 4.34]{HLN} and \cite[Proposition 4.5]{HLN},  we know that $n$-exact categories and $(n+2)$-angulated categories are $n$-exangulated categories.
There are some other examples of $n$-exangulated categories
 which are neither $n$-exact nor $(n+2)$-angulated, see \cite{HLN,HLN1,LZ,HZZ2}.
\end{example}

\section{An equivalent characterization of the axiom (EA2)}
In this section we introduce the definition of homotopy cartesian squares in pre-$n$-exangulated categories and provide an equivalent statement of the axiom (EA2).

\begin{definition}The following commutative diagram
$$\xymatrix{
 X_0 \ar[r]^{f_0}\ar[d]^{\varphi_0} & X_1 \ar[r]^{f_1}\ar[d]^{\varphi_1} & \cdots \ar[r]^{f_{n-2}}& X_{n-1}\ar[r]^{f_{n-1}}\ar[d]^{\varphi_{n-1}}& X_{n} \ar[d]^{\varphi_{n}} \\
 Y_0 \ar[r]^{g_0} & Y_1 \ar[r]^{g_1} & \cdots\ar[r]^{g_{n-2}}& Y_{n-1} \ar[r]^{g_{n-1}} & Y_{n} \\
}$$ in a pre-$n$-exangulated category $\C$ is called a $homotopy\ cartesian\ diagram$ if the following sequence
$$\xymatrixcolsep{4.5pc}\xymatrix{
X_0\ar[r]^{\left(
             \begin{smallmatrix}
               -f_0 \\
               \varphi_0 \\
             \end{smallmatrix}
           \right)
}& X_1\oplus Y_0\ar[r]^{\left(
             \begin{smallmatrix}
              -f_1 & 0 \\
               \varphi_1 & g_0 \\
             \end{smallmatrix}
           \right)}& X_2\oplus Y_1\ar[r]^{\left(
             \begin{smallmatrix}
              -f_2 & 0 \\
               \varphi_2 & g_1 \\
             \end{smallmatrix}
           \right)}& \cdots\\
}$$
$$\begin{gathered}\xymatrixcolsep{5pc}\xymatrix{
\cdots\ar[r]^{\left(
             \begin{smallmatrix}
              -f_{n-1} & 0 \\
               \varphi_{n-1} & g_{n-2} \\
             \end{smallmatrix}
           \right)\ \ \ \ }& X_{n}\oplus Y_{n-1}\ar[r]^{\ \ \ \left(
             \begin{smallmatrix}
               \varphi_{n} & g_{n-1} \\
             \end{smallmatrix}
           \right)}& Y_{n} \overset{\partial}{\dashrightarrow}
}\end{gathered}\eqno{(2.1)}$$
is a distinguished $n$-exangle, where $\partial$ is called a $differential$.
\end{definition}

\begin{remark}
The distinguished $n$-exangle (2.1) in the definition of homotopy cartesian diagram can also be written as follows
$$\xymatrixcolsep{4.5pc}\xymatrix{
X_0\ar[r]^{\left(
             \begin{smallmatrix}
               -f_0 \\
               \varphi_0 \\
             \end{smallmatrix}
           \right)
}& X_1\oplus Y_0\ar[r]^{\left(
             \begin{smallmatrix}
              f_1 & 0 \\
               \varphi_1 & g_0 \\
             \end{smallmatrix}
           \right)}& X_2\oplus Y_1\ar[r]^{\left(
             \begin{smallmatrix}
              f_2 & 0 \\
               -\varphi_2 & g_1 \\
             \end{smallmatrix}
           \right)}& \cdots\\
}$$
$$\begin{gathered}\xymatrixcolsep{5pc}\xymatrix{
\cdots\ar[r]^{\left(
             \begin{smallmatrix}
              f_{n-1} & 0 \\
               (-1)^n \varphi_{n-1} & g_{n-2} \\
             \end{smallmatrix}
           \right)\ \ \ \ }& X_{n}\oplus Y_{n-1}\ar[r]^{\ \ \ \left(
             \begin{smallmatrix}
              (-1)^{n+1} \varphi_{n} & g_{n-1} \\
             \end{smallmatrix}
           \right)}& Y_{n} \overset{\partial}{\dashrightarrow}.
}\end{gathered}\eqno{(2.2)}$$
In fact, the two distinguished $n$-exangles are isomorphic:
$$\xymatrixcolsep{4.5pc}\xymatrixrowsep{2.7pc}\xymatrix{
X_0\ar[r]^{\left(
             \begin{smallmatrix}
               -f_0 \\
               \varphi_0 \\
             \end{smallmatrix}
           \right)
}\ar@{=}[d]& X_1\oplus Y_0\ar[r]^{\left(
             \begin{smallmatrix}
              -f_1 & 0 \\
               \varphi_1 & g_0 \\
             \end{smallmatrix}
           \right)}\ar@{=}[d]& X_2\oplus Y_1\ar[r]^{\left(
             \begin{smallmatrix}
              -f_2 & 0 \\
               \varphi_2 & g_1 \\
             \end{smallmatrix}
           \right)}\ar[d]^{\left(
             \begin{smallmatrix}
              -1 & 0 \\
               0 & 1 \\
             \end{smallmatrix}
           \right)}& \cdots\\
X_0\ar[r]^{\left(
             \begin{smallmatrix}
               -f_0 \\
               \varphi_0 \\
             \end{smallmatrix}
           \right)
}& X_1\oplus Y_0\ar[r]^{\left(
             \begin{smallmatrix}
              f_1 & 0 \\
               \varphi_1 & g_0 \\
             \end{smallmatrix}
           \right)}& X_2\oplus Y_1\ar[r]^{\left(
             \begin{smallmatrix}
              f_2 & 0 \\
               -\varphi_2 & g_1 \\
             \end{smallmatrix}
           \right)}& \cdots\\
}$$
$$\xymatrixcolsep{5.5pc}\xymatrixrowsep{2.5pc}\xymatrix{
\cdots\ar[r]^{\left(
             \begin{smallmatrix}
              -f_{n-1} & 0 \\
               \varphi_{n-1} & g_{n-2} \\
             \end{smallmatrix}
           \right)}& X_{n}\oplus Y_{n-1}\ar[r]^{\left(
             \begin{smallmatrix}
               \varphi_{n} & g_{n-1} \\
             \end{smallmatrix}
           \right)}\ar[d]^{\left(
             \begin{smallmatrix}
              (-1)^{n+1} & 0 \\
              0 & 1 \\
             \end{smallmatrix}
           \right)} & Y_{n}\ar@{-->}[r]^-{\partial} \ar@{=}[d] & \\
\cdots\ar[r]^{\left(
             \begin{smallmatrix}
              f_{n-1} & 0 \\
               (-1)^n\varphi_{n-1} & g_{n-2} \\
             \end{smallmatrix}
           \right)}& X_{n}\oplus Y_{n-1}\ar[r]^{\left(
             \begin{smallmatrix}
               (-1)^{n+1}\varphi_{n} & g_{n-1} \\
             \end{smallmatrix}
           \right)}& Y_{n} \ar@{-->}[r]^\partial&.
}$$

\end{remark}

\begin{lemma}\emph{\cite[Proposition 3.6]{HLN}}\label{a2}
\rm Let $\C$ be an $n$-exangulated category, ${}_A\langle X_{\bullet},\delta\rangle_C$ and ${}_B\langle Y_{\bullet},\rho\rangle_D$ be distinguished $n$-exangles. Suppose that we are given a commutative square
$$\xymatrix{
 X_0 \ar[r]^{{d_0^X}} \ar@{}[dr]|{\circlearrowright} \ar[d]_{a} & X_1 \ar[d]^{b}\\
 Y_0  \ar[r]_{d_0^Y} &Y_1
}
$$
in $\C$. Then the following holds.

{\rm (1)}~ There is a morphism $f_{\bullet}\colon\Xd\to\langle Y_{\bullet},\rho\rangle$ which satisfies $f_0=a$ and $f_1=b$.

{\rm (2)}~ If $X_0=Y_0=A$ and $a=1_A$ for some $A\in\C$, then the above $f_{\bullet}$ can be taken to give a distinguished $n$-exangle $\langle M^f_{\bullet},(d^X_0)\sas\rho\rangle$.

\end{lemma}

\begin{theorem}\label{thm1}
Let $(\C,\E,\s)$ be a pre-$n$-exangulated category. Then $\C$ satisfies {\rm(EA2)} if and only if $\C$ satisfies {\rm(EA2-1)}:

Let $X_0\xrightarrow{f_0}X_1\xrightarrow{f_1}\cdots\xrightarrow{f_{n-2}}X_{n-1}
\xrightarrow{f_{n-1}}X_n\xrightarrow{f_n}X_{n+1}\overset{\delta}{\dashrightarrow}$ and $X_0\xrightarrow{g_0}Y_1\xrightarrow{g_1}\cdots\xrightarrow{g_{n-2}}Y_{n-1}
\xrightarrow{g_{n-1}}Y_n\xrightarrow{g_n}Y_{n+1}\overset{\delta^{'}}{\dashrightarrow}$ be distinguished $n$-exangles, and $\varphi_1:X_1\rightarrow Y_1$ be a morphism such that $\varphi_1f_0=g_0$.
Then there exist morphisms $\varphi_i:X_i\rightarrow Y_i$ for $2\leq i\leq n+1$, which give a morphism of distinguished $n$-exangles

$$\xymatrix{
X_0\ar[r]^{f_0}\ar@{}[dr] \ar@{=}[d] &X_1 \ar[r]^{f_1} \ar@{}[dr]\ar[d]^{\varphi_1}&X_2 \ar[r]^{f_2} \ar@{}[dr]\ar@{-->}[d]^{\varphi_2}&\cdot\cdot\cdot \ar[r]\ar@{}[dr] &X_n \ar[r]^{f_n} \ar@{}[dr]\ar@{-->}[d]^{\varphi_n}&X_{n+1} \ar@{}[dr]\ar@{-->}[d]^{\varphi_{n+1}} \ar@{-->}[r]^-{\delta} &\\
{X_0}\ar[r]^{g_0} &{Y_1}\ar[r]^{g_1}&{Y_2} \ar[r]^{g_2} &\cdot\cdot\cdot \ar[r] &{Y _n}\ar[r]^{g_n}  &{Y_{n+1}} \ar@{-->}[r]^-{\delta^{'}} &}
$$ and moreover, the following
$$\xymatrix{
 X_1 \ar[r]^{f_1}\ar[d]^{\varphi_1} & X_2 \ar[r]^{f_2}\ar[d]^{\varphi_2} & \cdots \ar[r]^{f_{n-1}}& X_{n}\ar[r]^{f_{n}}\ar[d]^{\varphi_{n}}& X_{n+1} \ar[d]^{\varphi_{n+1}} \\
 Y_1 \ar[r]^{g_1} & Y_2 \ar[r]^{g_2} & \cdots\ar[r]^{g_{n-1}}& Y_{n} \ar[r]^{g_{n}} & Y_{n+1} \\
}$$ is a homotopy cartesian diagram and $(f_0)_{*}{\delta^{'}}$ is the differential.

\end{theorem}

\begin{proof}
{\rm(EA2)}$ \Longrightarrow$ {\rm(EA2-1)}. It follows from Lemma \ref {a2}.

{\rm(EA2-1)}$ \Longrightarrow$ {\rm(EA2)}. Since $(\id,\varphi_{n+1}):(\varphi_{n+1})^{*}{\delta^{'}}\rightarrow\delta^{'}$ is a morphism of $\E$-extensions, by ${\rm (R0)}$, there exist morphisms $\varphi^{'}_i:X_i\rightarrow Y_i$ for $1\leq i\leq n$ such that the following diagram commutative
$$\xymatrix{
X_0\ar[r]^{f_0}\ar@{}[dr] \ar@{=}[d] &X_1 \ar[r]^{f_1} \ar@{}[dr]\ar@{-->}[d]^{\varphi^{'}_1}&X_2 \ar[r]^{f_2} \ar@{}[dr]\ar@{-->}[d]^{\varphi^{'}_2}&\cdot\cdot\cdot \ar[r]\ar@{}[dr] &X_n \ar[r]^{f_n} \ar@{}[dr]\ar@{-->}[d]^{\varphi^{'}_n}&X_{n+1} \ar@{}[dr]\ar[d]^{\varphi_{n+1}} \ar@{-->}[r]^-{(\varphi_{n+1})^{*}{\delta^{'}}} &\\
{X_0}\ar[r]^{g_0} &{Y_1}\ar[r]^{g_1}&{Y_2} \ar[r]^{g_2} &\cdot\cdot\cdot \ar[r] &{Y _n}\ar[r]^{g_n}  &{Y_{n+1}} \ar@{-->}[r]^-{\delta^{'}} &.}
$$
By {\rm(EA2-1)}, there exist morphisms $\varphi^{''}_i:X_i\rightarrow Y_i$ for $2\leq i\leq n+1$, which give a morphism of distinguished $n$-exangles
$$\xymatrix{
X_0\ar[r]^{f_0}\ar@{}[dr] \ar@{=}[d] &X_1 \ar[r]^{f_1} \ar@{}[dr]\ar[d]^{\varphi^{'}_1}&X_2 \ar[r]^{f_2} \ar@{}[dr]\ar@{-->}[d]^{\varphi^{''}_2}&\cdot\cdot\cdot \ar[r]\ar@{}[dr] &X_n \ar[r]^{f_n} \ar@{}[dr]\ar@{-->}[d]^{\varphi^{''}_n}&X_{n+1} \ar@{}[dr]\ar@{-->}[d]^{\varphi^{''}_{n+1}} \ar@{-->}[r]^-{(\varphi_{n+1})^{*}{\delta^{'}}} &\\
{X_0}\ar[r]^{g_0} &{Y_1}\ar[r]^{g_1}&{Y_2} \ar[r]^{g_2} &\cdot\cdot\cdot \ar[r] &{Y _n}\ar[r]^{g_n}  &{Y_{n+1}} \ar@{-->}[r]^-{\delta^{'}} &}
$$
and moreover, the following
$$\xymatrix{
 X_1 \ar[r]^{f_1}\ar[d]^{\varphi^{'}_1} & X_2 \ar[r]^{f_2}\ar[d]^{\varphi^{''}_2} & \cdots \ar[r]^{f_{n-1}}& X_{n}\ar[r]^{f_{n}}\ar[d]^{\varphi^{''}_{n}}& X_{n+1} \ar[d]^{\varphi^{''}_{n+1}} \\
 Y_1 \ar[r]^{g_1} & Y_2 \ar[r]^{g_2} & \cdots\ar[r]^{g_{n-1}}& Y_{n} \ar[r]^{g_{n}} & Y_{n+1} \\
}$$ is a homotopy cartesian diagram and $(f_0)_{*}{\delta^{'}}$ is the differential.

Since $(\varphi_{n+1})^{*}{\delta^{'}}=(\varphi^{''}_{n+1})^{*}{\delta^{'}}$, that is, $(\varphi_{n+1}-\varphi^{''}_{n+1})^{*}{\delta^{'}}=0$, then by the exactness of
$$
\C(X_{n+1},Y_{n})\xrightarrow{\C(X_{n+1},g_{n})}\C(X_{n+1},Y_{n+1})\xrightarrow{~\delta^{'}\ssh~}\E(X_{n+1},X_0),
$$
there is a morphism $h_{n+1}:X_{n+1}\rightarrow Y_{n}$ which gives $g_{n}h_{n+1}=\varphi_{n+1}-\varphi^{''}_{n+1}$.

Since   $g_{n}(\varphi^{'}_{n}-\varphi^{''}_{n}-h_{n+1}f_{n})=g_{n}\varphi^{'}_{n}-g_{n}\varphi^{''}_{n}-g_{n}h_{n+1}f_{n}=g_{n}\varphi^{'}_{n}-g_{n}\varphi^{''}_{n}-\varphi_{n+1}f_{n}+\varphi^{''}_{n+1}f_{n}=0$, then by the exactness of
$$
\C(X_{n},Y_{n-1})\xrightarrow{\C(X_{n},g_{n-1})}\C(X_{n},Y_{n})\xrightarrow{\C(X_{n},g_{n})}\C(X_{n},Y_{n+1}),
$$
there is a morphism $h_{n}:X_{n}\rightarrow Y_{n-1}$ which gives $g_{n-1}h_{n}=\varphi^{'}_{n}-\varphi^{''}_{n}-h_{n+1}f_{n}$.

Inductively, for $i=2,3,\cdots,n-1$, we obtain $h_{i}:X_{i}\rightarrow Y_{i-1}$ such that $g_{i-1}h_{i}=\varphi^{'}_{i}-\varphi^{''}_{i}-h_{i+1}f_{i}$. In particular, $g_{1}(\varphi^{'}_{1}-\varphi^{'}_{1}-h_{2}f_{1})=-g_{1}h_{2}f_{1}=-\varphi^{'}_{2}f_{1}+\varphi^{''}_{2}f_{1}+h_{3}f_{2}f_{1}=-g_{1}\varphi^{'}_{1}+g_{1}\varphi^{'}_{1}=0$, and hence there exist $h_{1}:X_{1}\rightarrow Y_{0}$ which gives $g_{0}h_{1}=-h_{2}f_{1}$.

Set \[{\varphi_{i}}=\begin{cases}

\varphi^{'}_{i}=\varphi^{''}_{i}+h_{i+1}f_{i}+g_{i-1}h_{i}&~~\text{if}~ i=2,3,\cdots,n\\

\varphi^{'}_{i}+h_{i+1}f_i&~~

\text{if} ~i=1

\end{cases},\]
then $\varphi_{1}f_0=(\varphi^{'}_{1}+h_{2}f_1)f_0=\varphi^{'}_{1}f_0=g_0$, $\varphi_{2}f_1=(\varphi^{''}_{2}+h_{3}f_2+g_{1}h_2)f_1=\varphi^{''}_{2}f_1+g_{1}h_2f_1=g_1\varphi^{'}_{1}+g_{1}h_2f_1=g_1(\varphi^{'}_{1}+h_2f_1)=g_1\varphi_{1}$. The following commutative diagram
$$\xymatrixcolsep{4.5pc}\xymatrixrowsep{2.7pc}\xymatrix{
X_1\ar[r]^{\left(
             \begin{smallmatrix}
               -f_1 \\
               \varphi^{'}_1 \\
             \end{smallmatrix}
           \right)
}\ar@{=}[d]& X_2\oplus Y_1\ar[r]^{\left(
             \begin{smallmatrix}
              -f_2 & 0 \\
               \varphi^{''}_2 & g_1 \\
             \end{smallmatrix}
           \right)}\ar[d]^{\left(
             \begin{smallmatrix}
              1 & 0 \\
              -h_2 & 1 \\
             \end{smallmatrix}
           \right)}& X_3\oplus Y_2\ar[r]^{\left(
             \begin{smallmatrix}
              -f_3 & 0 \\
               \varphi^{''}_3 & g_2 \\
             \end{smallmatrix}
           \right)}\ar[d]^{\left(
             \begin{smallmatrix}
             1 & 0 \\
              -h_3 & 1 \\
             \end{smallmatrix}
           \right)}& \cdots\\
X_1\ar[r]^{\left(
             \begin{smallmatrix}
               -f_1 \\
               \varphi_1 \\
             \end{smallmatrix}
           \right)
}& X_2\oplus Y_1\ar[r]^{\left(
             \begin{smallmatrix}
              -f_2 & 0 \\
               \varphi_2 & g_1 \\
             \end{smallmatrix}
           \right)}& X_3\oplus Y_2\ar[r]^{\left(
             \begin{smallmatrix}
              -f_3 & 0 \\
               \varphi_3 & g_2 \\
             \end{smallmatrix}
           \right)}& \cdots\\
}$$
$$\xymatrixcolsep{5.5pc}\xymatrixrowsep{2.5pc}\xymatrix{
\cdots\ar[r]^{\left(
             \begin{smallmatrix}
              -f_{n-1} & 0 \\
               \varphi^{''}_{n-1} & g_{n-2} \\
             \end{smallmatrix}
           \right)}& X_{n}\oplus Y_{n-1}\ar[d]^{\left(
             \begin{smallmatrix}
             1 & 0 \\
             -h_n & 1 \\
             \end{smallmatrix}
           \right)}\ar[r]^{\left(
             \begin{smallmatrix}
              -f_{n} & 0 \\
               \varphi^{''}_{n} & g_{n-1} \\
             \end{smallmatrix}
           \right)}& X_{n+1}\oplus Y_{n}\ar[r]^{\left(
             \begin{smallmatrix}
               \varphi^{''}_{n+1} & g_{n} \\
             \end{smallmatrix}
           \right)}\ar[d]^{\left(
             \begin{smallmatrix}
              1 & 0 \\
              -h_{n+1} & 1 \\
             \end{smallmatrix}
           \right)} & Y_{n+1} \ar@{=}[d] & \\
\cdots\ar[r]^{\left(
             \begin{smallmatrix}
              -f_{n-1} & 0 \\
               \varphi_{n-1} & g_{n-2} \\
             \end{smallmatrix}
           \right)}& X_{n}\oplus Y_{n-1}\ar[r]^{\left(
             \begin{smallmatrix}
              -f_{n} & 0 \\
               \varphi_{n} & g_{n-1} \\
             \end{smallmatrix}
           \right)}& X_{n+1}\oplus Y_{n}\ar[r]^{\left(
             \begin{smallmatrix}
               \varphi_{n+1} & g_{n} \\
             \end{smallmatrix}
           \right)}& Y_{n+1} &
}$$
implies that
$$\xymatrixcolsep{4.5pc}\xymatrix{
X_1\ar[r]^{\left(
             \begin{smallmatrix}
               -f_1 \\
               \varphi_1 \\
             \end{smallmatrix}
           \right)
}& X_2\oplus Y_1\ar[r]^{\left(
             \begin{smallmatrix}
              -f_2 & 0 \\
               \varphi_2 & g_1 \\
             \end{smallmatrix}
           \right)}& X_3\oplus Y_2\ar[r]^{\left(
             \begin{smallmatrix}
              -f_3 & 0 \\
               \varphi_3 & g_2 \\
             \end{smallmatrix}
           \right)}& \cdots\\
}$$
$$\begin{gathered}\xymatrixcolsep{5pc}\xymatrix{
\cdots\ar[r]^{\left(
             \begin{smallmatrix}
              -f_{n-1} & 0 \\
               \varphi_{n-1} & g_{n-2} \\
             \end{smallmatrix}
           \right)}& X_{n}\oplus Y_{n-1}\ar[r]^{\left(
             \begin{smallmatrix}
              -f_{n} & 0 \\
               \varphi_{n} & g_{n-1} \\
             \end{smallmatrix}
           \right)}& X_{n+1}\oplus Y_{n}\ar[r]^{\left(
             \begin{smallmatrix}
               \varphi_{n+1} & g_{n} \\
             \end{smallmatrix}
           \right)}& Y_{n+1} \overset{{(f_0)_{*}{\delta^{'}}}}{\dashrightarrow}
}\end{gathered}\eqno{}$$
is a distinguished $n$-exangle. Therefore, $(\id_{X_0}, \varphi_{n+1})$ has a good lift $(\id_{X_0},\varphi_{1},\varphi_{2},\cdots,\varphi_{n}, \varphi_{n+1})$. That is to say {\rm(EA2)} holds.
\end{proof}
\begin{corollary}\rm
In Theorem \ref{thm1}, when $n=1$, it is just the dual of Theorem 3.3 in
\cite{KLW}.
\end{corollary}

\section{Closed subfunctors arising from half exact functors}
In this section, we introduce methods of constructing closed subfunctors of an $n$-exangulated category $(\C,\E,\s)$ from half exact functors. Let's start with the following key lemma.

\begin{lemma}\label{lem1}{\rm\cite[Proposition 3.16]{HLN}}
Let $(\C,\E,\s)$ be an $n$-exangulated category. For any additive subfunctor $\F\subseteq\E$, the following statements are equivalent.

{\rm (1)}~ $(\C,\F,\s\hspace{-1.2mm}\mid_{\F})$ is an $n$-exangulated category.

{\rm (2)}~ $\s\hspace{-1.2mm}\mid_{\F}$-inflations are closed under composition.

{\rm (3)}~ $\s\hspace{-1.2mm}\mid_{\F}$-deflations are closed under composition.

\rm We call an additive subfunctor $\F\subseteq\E$ a closed subfunctor if it satisfies the above equivalent conditions. In this case, we call $(\C,\F,\s\hspace{-1.2mm}\mid_{\F})$ a relative theory of $(\C,\E,\s)$. For a relative theory $(\C,\F,\s\hspace{-1.2mm}\mid_{\F})$, we briefly denote it by $(\C,\F)$.
\end{lemma}

\begin{definition}\label{31} Let $\A$ be an abelian category. An additive functor $F:\C \rightarrow\A$ is called a half exact functor if the sequence $$FX_0\xrightarrow{Ff_0}FX_1\xrightarrow{Ff_1}\cdots\xrightarrow{Ff_{n-2}}FX_{n-1}
\xrightarrow{Ff_{n-1}}FX_n\xrightarrow{Ff_n}FX_{n+1}$$ is exact for any distinguished $n$-exangle
$X_0\xrightarrow{f_0}X_1\xrightarrow{f_1}\cdots\xrightarrow{f_{n-2}}X_{n-1}
\xrightarrow{f_{n-1}}X_n\xrightarrow{f_n}X_{n+1}\overset{\delta}{\dashrightarrow}.$ Moreover, if $Ff_n$ (resp. $Ff_0$) is an epimorphism (resp. monomorphism), we call $F$ a right exact functor (resp. left exact functor).
\end{definition}

\begin{remark}
If the category $\C$ are extriangulated, then Definition \ref{31} coincides with the definition
of half exact functor (homological functor) of extriangulated categories (cf. \cite{LN, Oga}).
\end{remark}

\begin{definition} Let $F:\C \rightarrow\A$ be a half exact functor. We define a subset $\E^{F}_{R}(X_{n+1},X_{0})$ of  $\E(X_{n+1},X_{0})$ consisting of $\delta$ such that for any distinguished $n$-exangle
$X_0\xrightarrow{f_0}X_1\xrightarrow{f_1}\cdots\xrightarrow{f_{n-2}}X_{n-1}
\xrightarrow{f_{n-1}}X_n\xrightarrow{f_n}X_{n+1}\overset{\delta}{\dashrightarrow}$, we have that $Ff_n$ is an epimorphism in $\A$. Similarly, we define a subset $\E^{F}_{L}(X_{n+1},X_{0})$ of $\E(X_{n+1},X_{0})$ consisting of $\delta$ such that $Ff_0$ is a monomorphism in $\A$.
\end{definition}
Note that the above definition is well-defined, that is, it does not depend on the choice of a distinguished $n$-exangle of $\delta$. Moreover, $\E^{F}_{R}(X_{n+1},X_{0})$ defines the maximum closed subfunctor such that F becomes a right exact functor.

\begin{theorem}\label{th1}
Let $F:\C \rightarrow\A$ be a half exact functor. Then the following statements hold.

{\rm (1)}~ $\E^{F}_{R}$ is a closed subfunctor of $\E$, hence $(\C,\E^{F}_{R})$ is a relative theory of $(\C,\E,\s)$.

{\rm (2)}~$F$ restricts to a right functor $F:(\C,\E^{F}_{R})\rightarrow \A$.

{\rm (3)}~Let $(\C,\F)$ be a relative theory of $(\C,\E,\s)$. If $F$ restricts to a right functor $F:(\C,\F)\rightarrow \A$, then we have $\F\subseteq\E^{F}_{R}$.

\end{theorem}
\begin{proof}
(1) First of all, we claim that $\E^{F}_{R}$ is a subfunctor of $\E$. In fact, for any $\delta\in \E^{F}_{R}(X_{n+1},X_{0})$ and any morphism $a_{0}:X_{0}\rightarrow Y_{0}$, by {(EA2$\op$)}, we obtain the following commutative diagram
$$\xymatrix{
X_0\ar[r]^{f_0}\ar@{}[dr] \ar[d]^{a_0} &X_1 \ar[r]^{f_1} \ar@{}[dr]\ar[d]^{a_1}&X_2 \ar[r]^{f_2} \ar@{}[dr]\ar[d]^{a_2}&\cdot\cdot\cdot \ar[r]\ar@{}[dr] &X_n \ar[r]^{f_n} \ar@{}[dr]\ar[d]^{a_n}&X_{n+1} \ar@{}[dr]\ar@{=}[d] \ar@{-->}[r]^-{\delta} &\\
{Y_0}\ar[r]^{g_0} &{Y_1}\ar[r]^{g_1}&{Y_2} \ar[r]^{g_2} &\cdot\cdot\cdot \ar[r] &{Y _n}\ar[r]^{g_n}  &{X_{n+1}} \ar@{-->}[r]^-{a_0\delta} &.}
$$
Applying $F$ to the above diagram, we have the following commutative diagram with exact rows
$$\xymatrix{
FX_0\ar[r]^{Ff_0}\ar@{}[dr] \ar[d]^{Fa_0} &X_1 \ar[r]^{Ff_1} \ar@{}[dr]\ar[d]^{Fa_1}&FX_2 \ar[r]^{Ff_2} \ar@{}[dr]\ar[d]^{Fa_2}&\cdot\cdot\cdot \ar[r]\ar@{}[dr] &FX_n \ar[r]^{Ff_n} \ar@{}[dr]\ar[d]^{Fa_n}&FX_{n+1} \ar@{}[dr]\ar@{=}[d]  &\\
{FY_0}\ar[r]^{Fg_0} &{FY_1}\ar[r]^{Fg_1}&{FY_2} \ar[r]^{Fg_2} &\cdot\cdot\cdot \ar[r] &{FY _n}\ar[r]^{Fg_n}  &{FX_{n+1}}&}
$$
in $\A$. Since $Ff_n$ is epimorphic, then $Fg_n$ is also epimorphic. So we have $a_0\delta\in\E^{F}_{R}(X_{n+1},Y_{0})$.

For any $\delta\in \E^{F}_{R}(X_{n+1},X_{0})$ and any morphism $a_{n+1}:Y_{n+1}\rightarrow X_{n+1}$, by {(EA2)}, we obtain the following commutative diagram
$$\xymatrix{
X_0\ar[r]^{g_0}\ar@{}[dr] \ar@{=}[d]  &Y_1 \ar[r]^{g_1} \ar@{}[dr]\ar[d]^{a_1}&Y_2 \ar[r]^{g_2} \ar@{}[dr]\ar[d]^{a_2}&\cdot\cdot\cdot \ar[r]\ar@{}[dr] &Y_n \ar[r]^{g_n} \ar@{}[dr]\ar[d]^{a_n}&Y_{n+1} \ar@{}[dr]\ar[d]^{a_{n+1}} \ar@{-->}[r]^-{\delta a_{n+1}} &\\
{X_0}\ar[r]^{f_0} &{X_1}\ar[r]^{f_1}&{X_2} \ar[r]^{f_2} &\cdot\cdot\cdot \ar[r] &{X _n}\ar[r]^{f_n}  &{X_{n+1}} \ar@{-->}[r]^-{\delta} &}
$$
such that
$$\xymatrixcolsep{4.5pc}\xymatrix{
Y_1\ar[r]^{\left(
             \begin{smallmatrix}
               -g_1 \\
               a_1 \\
             \end{smallmatrix}
           \right)
}& Y_2\oplus X_1\ar[r]^{\left(
             \begin{smallmatrix}
              -g_2 & 0 \\
               a_2 & f_1 \\
             \end{smallmatrix}
           \right)}& Y_3\oplus X_2\ar[r]^{\left(
             \begin{smallmatrix}
              -g_3 & 0 \\
               a_3 & f_2 \\
             \end{smallmatrix}
           \right)}& \cdots\\
}$$
$$\begin{gathered}\xymatrixcolsep{5pc}\xymatrix{
\cdots\ar[r]^{\left(
             \begin{smallmatrix}
              -g_{n-1} & 0 \\
               a_{n-1} & f_{n-2} \\
             \end{smallmatrix}
           \right)}& Y_{n}\oplus X_{n-1}\ar[r]^{\left(
             \begin{smallmatrix}
              -g_{n} & 0 \\
               a_{n} & f_{n-1} \\
             \end{smallmatrix}
           \right)}& Y_{n+1}\oplus X_{n}\ar[r]^{\left(
             \begin{smallmatrix}
              a_{n+1} & f_{n} \\
             \end{smallmatrix}
           \right)}& X_{n+1} \overset{{g_0{\delta^{}}}}{\dashrightarrow}
}\end{gathered}\eqno{}$$
is a distinguished $n$-exangle. Note that $g_0\delta\in\E^{F}_{R}(X_{n+1},Y_{1})$ by the previous argument, we have the following exact sequence
$$\xymatrixcolsep{4.5pc}\xymatrix{
FY_1\ar[r]^{F\left(
             \begin{smallmatrix}
               -g_1 \\
               a_1 \\
             \end{smallmatrix}
           \right)
}&F Y_2\oplus FX_1\ar[r]^{F\left(
             \begin{smallmatrix}
              -g_2 & 0 \\
               a_2 & f_1 \\
             \end{smallmatrix}
           \right)}& FY_3\oplus F X_2\ar[r]^{F\left(
             \begin{smallmatrix}
              -g_3 & 0 \\
               a_3 & f_2 \\
             \end{smallmatrix}
           \right)}& \cdots\\
}$$
$$\begin{gathered}\xymatrixcolsep{5pc}\xymatrix{
\cdots\ar[r]^{F\left(
             \begin{smallmatrix}
              -g_{n-1} & 0 \\
               a_{n-1} & f_{n-2} \\
             \end{smallmatrix}
           \right)}& FY_{n}\oplus F X_{n-1}\ar[r]^{F\left(
             \begin{smallmatrix}
              -g_{n} & 0 \\
               a_{n} & f_{n-1} \\
             \end{smallmatrix}
           \right)}&F Y_{n+1}\oplus F X_{n}\ar[r]^{F\left(
             \begin{smallmatrix}
              a_{n+1} & f_{n} \\
             \end{smallmatrix}
           \right)}&F X_{n+1}\ar[r]&0
}\end{gathered}\eqno{}$$
in $\A$. In order to prove that $Fg_{n}$ is an epimorphism in $\A$, suppose that $k\colon FY_{n+1}\to M$ is a morphism in $\A$ satisfying $k\circ Fg_{n}=0$. Let $(k,0 )\in\C(F Y_{n+1}\oplus F X_{n},M)$. Since $k\circ Fg_{n}=0$, then $(k,0 )\left(
                                            \begin{smallmatrix}
                                              -Fg_{n} & 0 \\
                                              Fa_{n} & Ff_{n-1}
                                            \end{smallmatrix}
                                          \right)=(-kFg_{n},0 )=0$. Thus there exists unique a morphism $s:FX_{n+1} \rightarrow M$, such that $s\circ(Fa_{n+1},Ff_{n})=(k,0 ) $. So $s\circ Fa_{n+1}=k$ and $s\circ Ff_{n}=0$. That is to say, we have the following  commutative diagram

$$\xymatrixcolsep{4.5pc}\xymatrix{
FY_{n}\oplus F X_{n-1}\ar[r]^{\left(
             \begin{smallmatrix}
              -Fg_{n} & 0 \\
               Fa_{n} & Ff_{n-1} \\
             \end{smallmatrix}
           \right)}\ar[dr]_{0}& F Y_{n+1}\oplus F X_{n}\ar[r]^{\left(
             \begin{smallmatrix}
              Fa_{n+1} & Ff_{n} \\
             \end{smallmatrix}
           \right)}\ar[d]^{(k,0)} &F X_{n+1}\ar@{-->}[dl]^{s}\ar[r]&0. \\
 &M&&
}
$$
Note that $Ff_{n}$ is an epimorphism, so we have $s=0$. Then $k=s\circ Fa_{n+1}=0$. This shows that $Fg_{n}$ is an epimorphism in $\A$, hence we have $\delta a_{n+1}\in\E^{F}_{R}$. Thus $\E^{F}_{R}$ is a subfunctor of $\E$.

Next, we need to show that $\E^{F}_{R}(X_{n+1},X_{0})$ is a subgroup of $\E(X_{n+1},X_{0})$. Note that $0\in\E^{F}_{R}(X_{n+1},X_{0})$, we only need to show $\delta^{'}-\delta\in\E^{F}_{R}(X_{n+1},X_{0})$ for any $\delta^{'},\delta\in\E^{F}_{R}(X_{n+1},X_{0})$. For $(-\id,\id):X_{0}\oplus X_{0}\rightarrow X_{0}$ and $\left(
                                            \begin{smallmatrix}
                                              \id \\
                                              \id \\
                                            \end{smallmatrix}
                                          \right):X_{n+1}\oplus X_{n+1}\rightarrow X_{n+1}$, we have
$\delta^{'}-\delta=(-\id,\id)(\delta\oplus\delta^{'})\left(
                                            \begin{smallmatrix}
                                              \id \\
                                              \id \\
                                            \end{smallmatrix}
                                          \right)$, it is enough to show that $\delta\oplus\delta^{'}\in\E^{F}_{R}(X_{n+1}\oplus X_{n+1},X_{0}\oplus X_{0})$, and this follows from Proposition 3.2 in \cite{HLN}. Thus $\E^{F}_{R}$ is an additive subfunctor of $\E$.

Finally, let $X_{n}\ov{f_{n}}{\lra}X_{n+1}\ov{{g_{n}}}{\lra}Y_{n+1}$ be any sequence of morphisms in $\C$.
Assume that $f_{n}$ and $g_{n}$ are $\s\hspace{-1.2mm}\mid_{\E^{F}_{R}}$-deflations. By (EA1), we know that $g_{n}\circ f_{n}$ is an $\s$-deflation. Thus we assume
$${\s}(\delta)=[Z_{0}\xrightarrow{~{h_{0}}~}
Z_{1}\xrightarrow{~{h_{1}}~}
Z_{2}\xrightarrow{~{h_{2}}~}\cdots\xrightarrow{{h_{n-2}}~}Z_{n-1}\xrightarrow{~{h_{n-1}}~}X_n\xrightarrow{{g_{n}f_{n}}~}Y_{n+1}].$$
Note that $Fg_{n}$ and $Ff_{n}$ are epimorphic, then $F(g_{n}f_{n})$ is an epimorphism, so $\delta\in\E^{F}_{R}$. This shows that $\s\hspace{-1.2mm}\mid_{\E^{F}_{R}}$-deflations are closed under composition. Thus $\E^{F}_{R}$ is a closed subfunctor of $\E$.

(2) and (3) follow immediately from the construction of $\E^{F}_{R}$.
\end{proof}

\begin{remark}
Dually, we obtain the statement with respect to $\E^{F}_{L}$ by applying the above proposition to a contravariant half exact functor $F:\C\rightarrow\A^{\rm op}$.
\end{remark}
\begin{example}
Let $\H\subseteq\C$ be a full subcategory. Define subfunctors $\E_{\H}$ and $\E^{\H}$ of $\E$ by
$$\E_{\H}(X_{n+1},X_0)=\{\delta\in\E(X_{n+1},X_0)|{(\delta_{\sharp})}_H=0~ \text{for any}~H\in\H\},$$
$$\E^{\H}(X_{n+1},X_0)=\{\delta\in\E(X_{n+1},X_0)|{(\delta^{\sharp})}_H=0~ \text{for any}~H\in\H\}.$$
In \cite[Proposition 3.17]{HLN}, it is shown that $\E_{\H}$ and $\E^{\H}$ are closed subfunctors of $\E$. They are special cases of Theorem \ref{th1}. Since the restricted Yoneda functors $Y_{\H}:\C\rightarrow{\rm Mod}\H$ and $Y^{\H}:\C\rightarrow{\rm Mod}\H^{op}$ are half exact functors, moreover, $\E_{\H}=\E^{Y^{\H}}_{R}$ and $\E^{\H}=\E^{Y^{\H}}_{L}$ hold. In fact, for any distinguished $n$-exangle
$$X_0\xrightarrow{f_0}X_1\xrightarrow{f_1}\cdots\xrightarrow{f_{n-2}}X_{n-1}
\xrightarrow{f_{n-1}}X_n\xrightarrow{f_n}X_{n+1}\overset{\delta}{\dashrightarrow},$$
we have the following exact sequence
$$
\C(-,X_0)|_\H\xrightarrow{Y_{\H}(f_0)}\cdots\xrightarrow{Y_{\H}(f_{n-1})}\C(-,X_{n})|_\H\xrightarrow{Y_{\H}(f_{n})}\C(-,X_{n+1})|_\H\xrightarrow{~\del\ssh~}\E(-,X_0)|_\H.
$$
Thus we have $\E^{Y^{\H}}_{R}(X_{n+1},X_0)=\E_{\H}(X_{n+1},X_0)$.

\end{example}
\begin{corollary}\rm
In Theorem \ref{th1}, when $n=1$, it is just Proposition A in
\cite{S}.
\end{corollary}

\textbf{Jian He}\\
Department of Mathematics, Nanjing University, 210093 Nanjing, Jiangsu, P. R. China\\
E-mail: \textsf{jianhe30@163.com}\\[0.3cm]
\textbf{Jing He}\\
College of Science, Hunan University of Technology and Business, 410205 Changsha, Hunan P. R. China\\
E-mail: \textsf{jinghe1003@163.com}\\[0.3cm]
\textbf{Panyue Zhou}\\
College of Mathematics, Hunan Institute of Science and Technology, 414006 Yueyang, Hunan, P. R. China.\\
E-mail: \textsf{panyuezhou@163.com}

\end{document}